\documentclass[preprint,12pt]{elsarticle}
\biboptions{numbers,sort&compress}
\usepackage[T1]{fontenc}
\usepackage{lmodern}
\usepackage{microtype}
\usepackage{amsmath,amssymb,amsfonts,amsthm}
\usepackage{booktabs,array,longtable}
\usepackage{enumitem}
\usepackage{hyperref}
\hypersetup{hidelinks,
 pdftitle={Torsion growth of rational elliptic curves over Zp-extensions of quadratic fields},
 pdfauthor={Shisong Xu and Haoyang Yuan}}
\journal{Journal of Number Theory}
\theoremstyle{plain}
\newtheorem{theorem}{Theorem}[section]
\newtheorem{proposition}[theorem]{Proposition}
\newtheorem{lemma}[theorem]{Lemma}
\newtheorem{corollary}[theorem]{Corollary}
\theoremstyle{definition}
\newtheorem{example}[theorem]{Example}
\theoremstyle{remark}
\newtheorem{remark}[theorem]{Remark}
\newcommand{\Q}{\mathbb Q}
\newcommand{\Z}{\mathbb Z}
\newcommand{\F}{\mathbb F}
\newcommand{\Gal}{\operatorname{Gal}}
\newcommand{\GL}{\operatorname{GL}}

\newcommand{\tors}{\mathrm{tors}}
\newcommand{\cyc}{\mathrm{cyc}}
\newcommand{\anti}{\mathrm{anti}}
\newcommand{\odd}{\mathrm{odd}}
\newcommand{\M}{\widetilde K_p}

\newcommand{\val}{\operatorname{v}}
\begin{document}
\begin{frontmatter}
\title{Torsion growth of rational elliptic curves over $\Z_p$-extensions of quadratic fields}
\author[nju]{Shisong Xu\corref{cor1}}
\ead{shsxu@smail.nju.edu.cn}
\author[nju]{Haoyang Yuan}
\ead{zgqyhy@163.com}
\cortext[cor1]{Corresponding author.}
\affiliation[nju]{organization={Department of Mathematics, Nanjing University},
 addressline={22 Hankou Road},city={Nanjing},postcode={210093},
 state={Jiangsu},country={People's Republic of China}}
\begin{abstract}
Let $E/\Q$ be an elliptic curve and let $\widetilde K_p$ be the compositum of all $\Z_p$-extensions of a quadratic field $K$. We prove that $E(\widetilde K_p)_{\tors}=E(K)_{\tors}$ for $p\geq5$. For $p=3$ and imaginary quadratic $K\neq\Q(\sqrt{-3})$, torsion on each extension is determined by its intersections with the cyclotomic extension and the $2$-division field. Over $\Q(\sqrt{-3})$, we construct infinitely many non-CM curves with full $3$-torsion in the first anticyclotomic layer and compute the $3$-primary torsion on every slope for eight CM curves. For $p=2$, we bound the odd-primary torsion and exclude all primes greater than $7$. We also give uniform bounds for non-CM primary torsion and correct two assertions in Li's preprint about noncyclotomic $\Z_3$-extensions.
\end{abstract}
\end{frontmatter}

\section{Introduction}

Let $E/\Q$ be an elliptic curve, $K$ a quadratic field, and $p$ a prime. We study torsion growth over $\M$, the compositum of all $\Z_p$-extensions of $K$, and over the individual extensions it contains. Distinct extensions can share finite layers. Moreover, torsion can be infinite over $\M$ and finite over every individual $\Z_p$-extension, as the CM examples below show. We write $C_n=\Z/n\Z$, with $C_1=0$.

Write $\Q_{\infty,p}$ for the cyclotomic $\Z_p$-extension of $\Q$ and $\Q_{r,p}$ for its layer of degree $p^r$. For real quadratic $K$, Leopoldt's conjecture is known and gives $\M=K\Q_{\infty,p}$. For imaginary quadratic $K$, $\Gal(\M/K)\simeq\Z_p^2$. When $p$ is odd, the eigenspaces of complex conjugation give the cyclotomic and anticyclotomic extensions and the decomposition
\[
 \M=K_{\cyc}K_{\anti}.
\]
We denote their layers of degree $p^r$ over $K$ by $K_{\cyc}^{(r)}$ and $K_{\anti}^{(r)}$, respectively, and use $r=\infty$ for the full extensions. The integral eigenspace decomposition above applies only to odd $p$.

\begin{theorem}\label{thm:main-large}
For every quadratic field $K$, every prime $p\geq5$, and every elliptic curve $E/\Q$,
\[
 E(\M)_{\tors}=E(K)_{\tors}.
\]
The same equality holds over each $\Z_p$-extension of $K$.
\end{theorem}

Avc{\i} proved the assertion for individual extensions when $p>5$ \cite{Avci}. The descent lemmas used there require the ambient field to be Galois over $\Q$, whereas a general $\Z_p$-extension of $K$ need not satisfy this condition. The compositum is Galois over $\Q$, so these lemmas apply to its finite Galois subfields. Cyclic torsion characters and quadratic twists then give a proof for all $p\geq5$.

For imaginary quadratic $K$ and $p=3$, choose generators $\gamma_c,\gamma_a$ of the cyclotomic and anticyclotomic eigenspaces. A slope $[\alpha:\beta]\in\mathbb P^1(\Z_3)$ defines a surjection $\varphi_{\alpha,\beta}:\Gal(\widetilde K_3/K)\to\Z_3$ by $\gamma_c\mapsto\alpha$ and $\gamma_a\mapsto\beta$, where at least one coordinate is a unit. Let $L_{\alpha,\beta}$ be the fixed field of its kernel.

\begin{theorem}\label{thm:main-slope}
Suppose $K$ is imaginary quadratic and $K\neq\Q(\sqrt{-3})$. For $L=L_{\alpha,\beta}$,
\[
 E(L)_{\tors}
 =E(L\cap K_{\cyc})_{\tors}^{\odd}
   \oplus E(L\cap K(E[2]))[2^\infty].
\]
If $\beta\neq0$, then
\[
 L\cap K_{\cyc}=K_{\cyc}^{(\val_3(\beta))},
\]
and $\beta=0$ gives $L=K_{\cyc}$. Theorem~\ref{thm:two-slope} describes the intersection with $K(E[2])$. If both $\alpha$ and $\beta$ are units, then $E(L)_{\tors}=E(K)_{\tors}$.
\end{theorem}

For the exceptional field and for $p=2$, we obtain the following results.

\begin{theorem}\label{thm:main-small}
\leavevmode
\begin{enumerate}[label=\textup{(\roman*)}]
\item For $K=\Q(\sqrt{-3})$, full $3$-torsion is contained in $\widetilde K_3$ if and only if $K(E[3])$ is $K$ or $K_{\anti}^{(1)}$. An infinite family of non-CM curves realizes the latter case \textup{(}Theorems~\ref{thm:full-three} and~\ref{thm:hesse-family}\textup{)}.
\item For the eight CM curves in Table~\ref{tab:cm-three}, all of $E[3^\infty]$ is contained in $\widetilde K_3$ for $K=\Q(\sqrt{-3})$, and Table~\ref{tab:cm-slopes} gives the exact finite group $E(L_{\alpha,\beta})[3^\infty]$ on every slope.
\item For any quadratic $K$, put $F=\widetilde K_2$. Then
\[
 \begin{split}
 E(F)[3^\infty]&\in\{0,C_3,C_9,C_3^2\},\\
 E(F)[5^\infty]&\in\{0,C_5\},\qquad
 E(F)[7^\infty]\in\{0,C_7\},\\
 E(F)[q^\infty]&=0\quad\text{for every prime }q\geq11.
 \end{split}
\]
The possible combinations of these primary groups remain to be determined.
\item If $E$ is non-CM, then
\[
 E(\widetilde K_2)[2^\infty]\subset E[2^{12}],
 \qquad
 E(\widetilde K_3)[3^\infty]\subset E[3^6].
\]
These bounds reduce the problem to finite division fields.
\end{enumerate}
\end{theorem}

The non-CM family in Theorem~\ref{thm:main-small}(i) shows that full $3$-torsion in a finite anticyclotomic layer does not force CM. A second family has full $4$-torsion over $\Q(i)$, with infinitely many non-CM $j$-invariants. These examples also show why good reduction cannot be deduced from the unramifiedness of a single finite torsion module.

\paragraph{Relation to Li's preprint}
Throughout, we cite the first arXiv version of Li's preprint \cite{Li}. That version also proves the $p\geq5$ compositum result and notes the Galois hypothesis in Avc{\i}'s argument. At $p=3$, however, its Lemma~4.2(2) and Proposition~4.8 rely on an incorrect assertion about intersections of distinct $\Z_3$-extensions. Section~\ref{subsec:Li} gives counterexamples to both statements and corrects the resulting claim in Theorem~1.1(ii). At $p=2$, Theorem~\ref{thm:two-odd-large} excludes $17$, left open in \cite{Li}, and supplies the missing argument for $13$; see Remark~\ref{rem:Li-two}.

\paragraph{Chronology}
The earlier manuscript, containing the $p\geq5$ theorem and the initial $p=3$ slope analysis, was submitted to the \emph{Journal of Number Theory} on 22 June 2026; receipt was acknowledged on 23 June under manuscript number \texttt{JNTH-D-26-00825}. Li's version~1 was posted on 15 July 2026. The additional small-prime results and the comparisons with \cite{Li} were developed for the present revision, dated 9 September 2026.

\section{Galois preliminaries and torsion characters}\label{sec:general}

Fix an algebraic closure of $\Q$ containing all fields under consideration, and write $G_F$ for the absolute Galois group of $F$. We use the theory of quadratic $\Z_p$-extensions in \cite{Iwasawa,Washington,Brink} and Brumer's theorem for real quadratic fields \cite{Brumer}. Since each element of $G_\Q$ permutes the $\Z_p$-extensions of $K$, the compositum $\M$ is Galois over $\Q$. It is unramified outside $p$ over $K$. Hence, for every rational prime $\ell\neq p$ and every finite subfield $N$ of $\M$,
\begin{equation}\label{eq:inertia-bound}
 e_\ell(N/\Q)\leq2.
\end{equation}
For imaginary quadratic $K$ and odd $p$, conjugation acts by $+1$ and $-1$ on the two rank-one eigenspaces. Every point of $E(\M)_{\tors}$ is defined over a finite Galois subfield of the form
\[
 K_{\cyc}^{(r)}K_{\anti}^{(s)}.
\]
This supplies finite Galois extensions of $\Q$ on which to apply descent.

We use Mazur's torsion theorem and the rational cyclic-isogeny list \cite{Mazur77,Mazur,Kenku}: a rational cyclic isogeny has degree
\begin{equation}\label{eq:isogeny-list}
 n\leq19\quad\text{or}\quad n\in\{21,25,27,37,43,67,163\}.
\end{equation}
We also use the classifications of torsion on rational elliptic curves over quadratic and cubic fields \cite{Najman}, and over quartic Galois fields \cite{ChouQuartic}.

\begin{lemma}[Coprime primary control]\label{lem:coprime}
Let $F/K$ be a Galois pro-$p$ extension and let $q\neq p$ be prime. Then
\[
 E(F)[q^\infty]=E(F\cap K(E[q]))[q^\infty].
\]
\end{lemma}
\begin{proof}
Put $D=K(E[q])$. The extension $D(E[q^n])/D$ has $q$-power degree, since its Galois group lies in the kernel of reduction from $\GL_2(\Z/q^n\Z)$ to $\GL_2(\F_q)$. The extension $DF/D$ is pro-$p$. Their intersection is $D$, so every $q^n$-torsion point defined over $F$ is defined over $D$. Its coordinates therefore lie in $F\cap D$.
\end{proof}

\begin{lemma}[Cyclic characters]\label{lem:cyclic-character}
Let $p$ be odd, and let $q$ be an odd prime such that $E(\M)[q]\neq E[q]$. Then
\[
 E(\M)[q^\infty]=E(K_{\cyc})[q^\infty].
\]
For every subfield $L$ of $\M$ containing $K$, the corresponding group over $L$ is defined over $L\cap K_{\cyc}$.
\end{lemma}
\begin{proof}
Let $P\in E(\M)$ have order $q^n$. Since the $q$-rank is at most one, $E(\M)[q^n]=\langle P\rangle$. This subgroup is $G_\Q$-stable because $\M/\Q$ is Galois. The action on it is given by a character with values in $(\Z/q^n\Z)^\times$.

For imaginary quadratic $K$, write $A=\Gal(\M/K_{\cyc})$. Conjugation inverts $A$, whereas any character has abelian image. Therefore the image of $A$ under this character has exponent dividing $2$. Since $A$ is pro-$p$ and $p$ is odd, that image is trivial. Thus $P\in E(K_{\cyc})$. For real quadratic $K$, we have $\M=K_{\cyc}$. The assertion over $L$ follows by intersecting the field of definition of each point with $L$.
\end{proof}

For $K=\Q(\sqrt d)$, $p$ odd, and $B=\Q_{r,p}$, where $0\leq r\leq\infty$, the usual quadratic-twist decomposition gives, for every odd $n$,
\begin{equation}\label{eq:twist}
 E(KB)[n]\simeq E(B)[n]\oplus E^d(B)[n].
\end{equation}
Indeed, the idempotents $(1\pm\sigma)/2$ are defined on odd torsion. By Chou--Daniels--Krijan--Najman \cite{CDKN}, torsion over $\Q_{\infty,p}$ equals rational torsion for $p\geq5$; for $p=2$ the possible groups form Mazur's list; for $p=3$ that list is augmented by $C_{21}$ and $C_{27}$.

\begin{proof}[Proof of Theorem~\ref{thm:main-large}]
For $q\geq5$, full $q$-torsion over $\M$ would imply $\mu_q\subset\M$. If $q\neq p$, this contradicts \eqref{eq:inertia-bound}, since $e_q(\Q(\zeta_q)/\Q)=q-1>2$. If $q=p$, the nontrivial prime-to-$p$ degree $[K(\zeta_p):K]$ contradicts the pro-$p$ condition. Thus Lemma~\ref{lem:cyclic-character} applies. It also applies to $q=3$ unless $K=\Q(\sqrt{-3})$, since an odd-degree extension of $K$ cannot acquire $\mu_3$.

For $q=2$, and also for $q=3$ in the exceptional quadratic field, the group $\GL_2(\F_q)$ has order divisible only by $2$ and $3$. The field $K(E[q])$ therefore intersects $\M$ only in $K$. Lemma~\ref{lem:coprime} gives no growth for these parts. All other odd torsion is cyclotomic by Lemma~\ref{lem:cyclic-character}, and \eqref{eq:twist}, applied to $E$ and $E^d$, together with \cite{CDKN}, shows that it is already defined over $K$.
\end{proof}

\begin{lemma}[A commutator bound]\label{lem:commutator}
Let $E$ be defined over a number field $K$, and let $A/K$ be an abelian extension. Suppose the image of $G_K$ on $T_q(E)$ contains
\[
 U=\begin{pmatrix}1&t\\0&1\end{pmatrix},\qquad
 V=\begin{pmatrix}1&0\\t&1\end{pmatrix},\quad t=q^a.
\]
Then $E(A)[q^\infty]\subset E[q^{2a}]$.
\end{lemma}
\begin{proof}
Commutators of lifts of $U,V$ fix $E(A)$. Direct multiplication gives
\[
 UVU^{-1}V^{-1}-I
 =t^2\begin{pmatrix}1+t^2&-t\\t&-1\end{pmatrix}.
\]
The matrix multiplying $t^2$ has determinant $-1$ and is invertible over $\Z_q$. A fixed point in $(\Q_q/\Z_q)^2$ is therefore killed by $t^2$.
\end{proof}

\section{Torsion on \texorpdfstring{$\Z_3$}{Z3}-extensions}\label{sec:slopes}

Throughout this section, $K$ is imaginary quadratic and $M=\widetilde K_3$. With the notation of the introduction, $\Gal(M/L_{\alpha,\beta})$ is generated by $\beta\gamma_c-\alpha\gamma_a$. We abbreviate $L=L_{\alpha,\beta}$.

\begin{lemma}[Axis intersections]\label{lem:axis}
For $\beta\neq0$ and $\alpha\neq0$, respectively,
\[
 L\cap K_{\cyc}=K_{\cyc}^{(\val_3(\beta))},\qquad
 L\cap K_{\anti}=K_{\anti}^{(\val_3(\alpha))}.
\]
If $\beta=0$ then $L=K_{\cyc}$; if $\alpha=0$ then $L=K_{\anti}$.
\end{lemma}
\begin{proof}
The Galois group of the cyclotomic intersection is
\[
 \Gal(M/K)/(\ker\varphi_{\alpha,\beta}+\Z_3\gamma_a)
 \simeq\Z_3/\beta\Z_3.
\]
The other assertion follows by interchanging the two coordinates.
\end{proof}

\begin{proposition}\label{prop:odd-three}
If $K\neq\Q(\sqrt{-3})$, then
\[
 E(M)_{\tors}^{\odd}=E(K_{\cyc})_{\tors}^{\odd},\qquad
 E(L)_{\tors}^{\odd}=E(L\cap K_{\cyc})_{\tors}^{\odd}.
\]
For $K=\Q(\sqrt{-3})$ the same assertions hold prime by prime for every odd prime other than $3$.
\end{proposition}
\begin{proof}
For $q\geq5$, \eqref{eq:inertia-bound} excludes $\mu_q$ from $M$. For $q=3$ and $K\neq\Q(\sqrt{-3})$, the pro-$3$ condition excludes $\mu_3$. The Weil pairing and Lemma~\ref{lem:cyclic-character} now apply.
\end{proof}

\begin{theorem}[$2$-division slope criterion]\label{thm:two-slope}
For every imaginary quadratic $K$,
\[
 E(L)[2^\infty]=E(L\cap K(E[2]))[2^\infty].
\]
The intersection equals $K$ unless one of the following holds:
\begin{enumerate}[label=\textup{(\roman*)}]
\item $\Q(E[2])=\Q(\zeta_9)^+$ and $\beta\equiv0\pmod3$; the intersection is $K_{\cyc}^{(1)}$.
\item $K(E[2])=K_{\anti}^{(1)}$ and $\alpha\equiv0\pmod3$; the intersection is $K_{\anti}^{(1)}$.
\end{enumerate}
The displayed equality also holds with $L$ replaced by $M$.
\end{theorem}
\begin{proof}
The equality follows from Lemma~\ref{lem:coprime}. Put $D=K(E[2])$. Its Galois group over $K$ embeds into $S_3$. A nontrivial pro-$3$ quotient of this group exists only when the group itself is $C_3$; $S_3$ has no quotient of order $3$. Thus a nontrivial intersection forces $D/K$ to be cyclic cubic and $D\subset M$.

Since $D/\Q$ is Galois, the kernel of the quotient $\Gal(M/K)\to\Gal(D/K)$ is stable under conjugation. Put $\Gamma=\Gal(M/K)$. The only conjugation-invariant lines in $\Gamma/3\Gamma$ are the two eigenspaces. Hence $D$ is $K_{\cyc}^{(1)}$ or $K_{\anti}^{(1)}$.

In the cyclotomic case put $S=\Q(E[2])$. The group $\Gal(K_{\cyc}^{(1)}/\Q)\simeq C_6$ has only cyclic quotients, whereas $\Gal(S/\Q)$ embeds in $S_3$. As $KS=K_{\cyc}^{(1)}$, $S$ must be its unique cubic subfield, $\Q(\zeta_9)^+$. In the anticyclotomic case $S=K_{\anti}^{(1)}$ is an $S_3$ splitting field with quadratic resolvent $K$. Lemma~\ref{lem:axis} translates the two containments into the indicated congruences. The converses follow at once.
\end{proof}

\begin{proof}[Proof of Theorem~\ref{thm:main-slope}]
Apply Proposition~\ref{prop:odd-three} and Theorem~\ref{thm:two-slope} to the odd and $2$-primary parts, respectively. Lemma~\ref{lem:axis} gives the cyclotomic intersection. When both coordinates are units, the two intersections equal $K$.
\end{proof}

\begin{corollary}\label{cor:sharp-two}
If $K\neq\Q(i)$ and $L\cap K(E[2])\neq K$, then $E(L)[2^\infty]=E[2]$.
\end{corollary}
\begin{proof}
The cyclic cubic group permutes the three nonzero $2$-torsion points transitively. If a point $P$ of order $4$ existed over $D=K(E[2])$, then $P$ and one conjugate would double to independent points of $E[2]$ and hence generate $E[4]$. The Weil pairing would put $i$ in $D$, contradicting $[D:K]=3$ and $K\neq\Q(i)$.
\end{proof}

\begin{corollary}\label{cor:odd-list}
For $K\neq\Q(\sqrt{-3})$, the odd torsion over $K_{\cyc}^{(r)}$ is cyclic and its order belongs to
\begin{center}
\begin{tabular}{ll}
\toprule
Layer & Possible odd orders\\
\midrule
$r=0$ & $1,3,5,7,9,15$\\
$r=1$ & $1,3,5,7,9,15,21$\\
$r\geq2$, including $r=\infty$ & $1,3,5,7,9,15,21,27$\\
\bottomrule
\end{tabular}
\end{center}
In particular, $13$ does not occur, and a point of order $27$ on a noncyclotomic slope requires $\val_3(\beta)\geq2$.
\end{corollary}
\begin{proof}
Equation~\eqref{eq:twist} and \cite{CDKN} leave only the primes $3,5,7$, with exponents at most $27,5,7$. The Weil pairing shows each primary component is cyclic. Their product is $G_\Q$-stable and gives a rational cyclic isogeny; \eqref{eq:isogeny-list} excludes the unlisted products. The quadratic classification gives the first row. In the first cyclotomic layer, a point of order $27$ would yield such a point on $E$ or its quadratic twist over a cubic field, contrary to \cite{Najman}. Lemma~\ref{lem:axis} gives the condition on $\beta$. The listed orders are necessary bounds; some may not occur for a fixed $K$.
\end{proof}

\begin{example}
The curve $y^2=x^3-3x+1$ has $2$-division field $\Q(\zeta_9)^+$: its cubic is irreducible with discriminant $81$. Thus every slope with $\beta\equiv0\pmod3$ acquires its full $2$-torsion. For $K=\Q(i)$, the splitting field of $x^3-3x-4$ is $K_{\anti}^{(1)}$ \cite{Brink}; the curve $y^2=x^3-3x-4$ illustrates the other residue class in Theorem~\ref{thm:two-slope}.
\end{example}

\subsection{Corrections to Li's version 1}\label{subsec:Li}

Distinct rank-one quotients of $\Z_3^2$ can have the same reduction modulo $3^r$. For example,
\begin{equation}\label{eq:noncyc-intersection}
 L_{1,3}\neq K_{\cyc},\qquad
 L_{1,3}\cap K_{\cyc}=K_{\cyc}^{(1)}.
\end{equation}
The proof of Lemma~4.2(2) in \cite[p.~12]{Li} asserts that distinct extensions intersect only in $K$, contrary to \eqref{eq:noncyc-intersection}. The roots of unity in $L_{1,3}$ also contradict the statement of that lemma.

\begin{proposition}[Roots of unity]\label{prop:roots}
For $K=\Q(\sqrt{-3})$ and $\beta\neq0$,
\[
 \mu(L_{\alpha,\beta})=\mu_{2\cdot3^{\val_3(\beta)+1}}.
\]
For $\beta=0$ the group is $\mu_{2\cdot3^\infty}$.
\end{proposition}
\begin{proof}
The pro-$3$ condition excludes $i$, and \eqref{eq:inertia-bound} excludes roots of unity of prime order at least $5$. All $3$-power roots of unity lie on $K_{\cyc}=K(\mu_{3^\infty})$, whose $r$-th layer is $K(\mu_{3^{r+1}})$. Apply Lemma~\ref{lem:axis}.
\end{proof}

In particular, $L_{1,3}$ contains $\mu_9$ but not $\mu_{27}$, and its roots of unity form $\mu_{18}$. This contradicts the dichotomy in \cite[Lemma~4.2(2)]{Li} between no new roots of unity and containment of the entire cyclotomic extension.

\begin{example}[Counterexample to Li's Proposition 4.8]\label{ex:seven}
Let $K=\Q(i)$ and
\[
 E:y^2+xy+y=x^3-x^2-5x+5.
\]
Let $u^3-3u+1=0$ and $B=\Q(u)=\Q(\zeta_9)^+$. The point
\[
 P=(3-u-u^2,\;1-u-u^2)
\]
has order $7$ on $E$, by direct calculation with the group law. Its $x$-coordinate is not rational, so $\Q(P)=B$. Since $B\cap K=\Q$, the point is not defined over $K$, whereas \eqref{eq:noncyc-intersection} gives $P\in E(L_{1,3})$. Thus
\[
 E(L_{1,3})[7^\infty]\neq E(K)[7^\infty]
\]
although $L_{1,3}/K$ is noncyclotomic. Computation over these number fields gives $E(K)_{\tors}=C_3$ and $E(B)_{\tors}=C_{21}$.
\end{example}

The proof of \cite[Proposition~4.8, p.~14]{Li} infers that a noncyclotomic extension cannot contain a nontrivial finite cyclotomic layer. Proposition~4.7 of that paper identifies a cyclotomic field of definition, but \eqref{eq:noncyc-intersection} shows why the inference fails. The correct formula, also valid for $K=\Q(\sqrt{-3})$, is
\[
 E(L)[7^\infty]=E(L\cap K_{\cyc})[7^\infty].
\]
The assertion on $7$-primary torsion in \cite[Theorem~1.1(ii)]{Li} requires the same correction.

\section{The exceptional field \texorpdfstring{$\Q(\sqrt{-3})$}{Q(sqrt(-3))}}\label{sec:exceptional}

We begin with a class field calculation that also applies to two quadratic fields at $p=2$.

\begin{lemma}[Ray class description of the compositum]\label{lem:ray-compositum}
For
\[
 (K,p)=(\Q(\sqrt{-3}),3),\quad(\Q(i),2),\quad(\Q(\sqrt{-2}),2),
\]
the maximal abelian extension of $K$ unramified outside $p$ equals $\M$, and its Galois group is $\Z_p^2$.
\end{lemma}
\begin{proof}
Each field has class number one and a unique prime $\mathfrak p$ above $p$. The inverse limit of its ray class groups of $\mathfrak p$-power conductor is
\[
 \mathcal O_{K_{\mathfrak p}}^\times/
       \overline{\mathcal O_K^\times}.
\]
In these three cases the torsion subgroup of the local unit group consists exactly of the global roots of unity, of orders $6,4,2$, respectively. The quotient is therefore torsion-free of $\Z_p$-rank two. There is no prime-to-$p$ factor. Class field theory therefore identifies the maximal abelian extension unramified outside $p$ with a $\Z_p^2$-extension. Its rank-one quotients, together with the fact that every $\Z_p$-extension is unramified outside $p$, identify it with $\M$.
\end{proof}

\begin{corollary}[Full primary torsion for CM curves]\label{cor:cm-sufficient}
For any of the three pairs in Lemma~\ref{lem:ray-compositum}, suppose $E/\Q$ has complex multiplication by an order in $K$ and has good reduction over $K$ outside $p$. Then
\[
 E[p^\infty]\subset E(\M).
\]
\end{corollary}
\begin{proof}
All geometric endomorphisms are defined over $K$. The $p$-adic image over $K$ centralizes their action, and is therefore contained in the multiplicative group of the quadratic $\Q_p$-algebra $K\otimes\Q_p$; in particular it is abelian. Good reduction outside $p$ implies that this representation is unramified there by the N\'eron--Ogg--Shafarevich criterion \cite{Silverman}. Lemma~\ref{lem:ray-compositum} gives the containment, also when the CM order is nonmaximal.
\end{proof}

In the rest of this section put $K=\Q(\sqrt{-3})$ and $M=\widetilde K_3$.

\subsection{Full \texorpdfstring{$3$}{3}-torsion and a non-CM family}

\begin{theorem}[Full $3$-torsion criterion]\label{thm:full-three}
One has $E[3]\subset E(M)$ if and only if
\[
 K(E[3])=K\quad\text{or}\quad K(E[3])=K_{\anti}^{(1)}.
\]
In the second case,
\[
 E[3]\subset E(L_{\alpha,\beta})
 \quad\Longleftrightarrow\quad \alpha\equiv0\pmod3.
\]
In particular, $K(E[3])$ cannot equal $K_{\cyc}^{(1)}$.
\end{theorem}
\begin{proof}
Suppose $D=K(E[3])\subset M$. The mod-$3$ image of $G_K$ is a $3$-subgroup of $\GL_2(\F_3)$, hence is trivial or cyclic of order $3$. In the latter case it is conjugate to the upper unipotent subgroup. Complex conjugation has eigenvalues $1,-1$ on $E[3]$. As it normalizes the unipotent subgroup, it inverts that subgroup. Thus $D/K$ is an anticyclotomic cubic quotient of $\Gal(M/K)$ and $D=K_{\anti}^{(1)}$. The converse is immediate, and the slope condition follows from Lemma~\ref{lem:axis}.
\end{proof}

\begin{theorem}[A family with fixed $3$-division field]\label{thm:hesse-family}
Let $r=\sqrt[3]{3}$ be real. For $t\in\Q$ define
\begin{equation}\label{eq:hesse-family}
 E_t:\ y^2=x^3-27t(t^3+24)x+54(t^6-60t^3-72).
\end{equation}
Then
\[
 \Q(E_t[3])=\Q(\sqrt{-3},r)=K_{\anti}^{(1)}.
\]
Consequently, $E_t[3]$ is defined over $L_{\alpha,\beta}$ exactly when $\alpha\equiv0\pmod3$. The family contains infinitely many non-CM $j$-invariants.
\end{theorem}
\begin{proof}
The discriminant and $j$-invariant are
\begin{equation}\label{eq:hesse-invariants}
 \Delta(E_t)=241864704(t^3-3)^3,
 \qquad j(E_t)=\frac{9t^3(t^3+24)^3}{(t^3-3)^3}.
\end{equation}
Since $t^3\neq3$ for rational $t$, every $E_t$ is nonsingular. Consider
\[
 \begin{split}
 P_t&=(-9t^2,\;12(t^3-3)\sqrt{-3}),\\
 Q_t&=(3t^2+12tr+12r^2,\;36r(t^2+tr+r^2)).
 \end{split}
\]
Substitution using $r^3=3$ verifies the curve equation and
\[
 \psi_3(x)=3x^4+6Ax^2+12Bx-A^2=0
\]
at both $x$-coordinates, where $A,B$ are the coefficients in \eqref{eq:hesse-family}. Thus both points have order $3$. Their $x$-coordinates differ by $12(t^2+tr+r^2)\neq0$, so they are independent and generate $E_t[3]$.

The $y$-coordinate of $P_t$ generates $K$, while the $x$-coordinate of $Q_t$ has degree $3$ over $\Q$: its coefficient of $r^2$ is $12$, so it is not rational in the basis $1,r,r^2$. Hence the full division field is exactly $S=K(r)$. This is the splitting field of $x^3-3$; $S/K$ is cyclic cubic and unramified outside $3$. By Lemma~\ref{lem:ray-compositum}, $S\subset M$. Conjugation inverts $\Gal(S/K)$, giving $S=K_{\anti}^{(1)}$.

If $t$ is an integer prime to $3$, then $t^3-3$ is prime to $3$, to $t$, and to $t^3+24$. Its absolute value exceeds $1$. Thus \eqref{eq:hesse-invariants} is not an integer. A rational CM $j$-invariant is integral, so these curves are non-CM. The nonconstant rational function $j(E_t)$ has finite fibers, so these curves have infinitely many distinct $j$-invariants.
\end{proof}

For related families with prescribed $3$-torsion representations obtained from Hesse pencils, see \cite{Hesse}. The points above give the division field of \eqref{eq:hesse-family} directly.

\begin{example}\label{ex:noncm-three}
At $t=1$ we obtain
\[
 E:y^2=x^3-675x-7074,
 \qquad j(E)=-140625/8.
\]
A minimal equation has coefficients $[1,-1,0,-42,-100]$ and conductor $162$, and
$E(K)_{\tors}=C_3$ and $E(K_{\anti}^{(1)})_{\tors}=C_3^2$.
The nonintegral $j$-invariant shows that $E$ is non-CM. Its negative valuation above $2$ also gives bad reduction there after base change to $K$. Thus $E[3]$ is unramified outside $3$ over $K$, although $E$ has bad reduction at $2$.
\end{example}

\subsection{A uniform bound for non-CM \texorpdfstring{$3$}{3}-primary torsion}

\begin{theorem}\label{thm:noncm-three-bound}
Let $K$ be any quadratic field and let $E/\Q$ be non-CM. Then
\[
 E(\widetilde K_3)[3^\infty]\subset E[3^6].
\]
Consequently, for every $K\subset L\subset\widetilde K_3$,
\[
 E(L)[3^\infty]
 =E(L\cap K(E[3^6]))[3^6].
\]
\end{theorem}
\begin{proof}
If $E(\widetilde K_3)[3]$ has rank at most one, a point of order $3^n$ generates a $G_\Q$-stable cyclic subgroup, as in Lemma~\ref{lem:cyclic-character}. The isogeny list \eqref{eq:isogeny-list} forces $n\leq3$.

Otherwise the mod-$3$ image of $G_K$ is trivial or has order $3$. In the nontrivial case its unique fixed line is stable under $G_\Q$, since $G_K$ is normal. In the trivial case the image of $G_\Q$ has order at most $2$ and is diagonalizable over $\F_3$. In both cases, $E$ has a rational cyclic $3$-isogeny.

By \cite[Corollary~1.1]{RSZB}, the $3$-adic image of a non-CM rational curve with a rational $3$-isogeny contains every matrix congruent to $I$ modulo $27$. This pro-$3$ subgroup also lies in the image of $G_K$, which is normal of index at most $2$. Apply Lemma~\ref{lem:commutator} with $t=27$ and the abelian extension $\widetilde K_3/K$. Intersecting coordinate fields gives the last equality.
\end{proof}

Together with Proposition~\ref{prop:odd-three}, \eqref{eq:twist}, and Theorem~\ref{thm:two-slope}, this bound shows that $E_t(M)_{\tors}$ is finite for every non-CM member of the family in Theorem~\ref{thm:hesse-family}.

\subsection{Torsion on slopes for eight CM curves}

Table~\ref{tab:cm-three} lists the eight CM curves considered in this subsection. We use the notation $[a_1,a_2,a_3,a_4,a_6]$ for their Weierstrass coefficients.

\begin{table}[htbp]
\centering
\caption{CM curves with full $3$-power torsion over $M$.}\label{tab:cm-three}
\begin{tabular}{clrr}
\toprule
Type & Coefficients & Conductor over $\Q$ & $j$\\
\midrule
A & $[0,0,1,0,0]$ & $27$ & $0$\\
A & $[0,0,1,0,-7]$ & $27$ & $0$\\
B & $[0,0,1,-30,63]$ & $27$ & $-12288000$\\
B & $[0,0,1,-270,-1708]$ & $27$ & $-12288000$\\
C & $[0,0,1,0,-1]$ & $243$ & $0$\\
C & $[0,0,1,0,2]$ & $243$ & $0$\\
C & $[0,0,1,0,20]$ & $243$ & $0$\\
C & $[0,0,1,0,-61]$ & $243$ & $0$\\
\bottomrule
\end{tabular}
\end{table}

\begin{theorem}[CM torsion on every slope]\label{thm:cm-slopes}
For each curve in Table~\ref{tab:cm-three},
\[
 E(M)[3^\infty]=E[3^\infty]\simeq(\Q_3/\Z_3)^2.
\]
For every slope $[\alpha:\beta]$, the group $E(L_{\alpha,\beta})[3^\infty]$ is finite and is given by Table~\ref{tab:cm-slopes}. The congruence conditions are unchanged if either chosen generator is rescaled by a unit.
\end{theorem}

\begin{table}[htbp]
\centering
\caption{Groups $E(L_{\alpha,\beta})[3^\infty]$ for the curves in Table~\ref{tab:cm-three}.}\label{tab:cm-slopes}
\begin{tabular}{cccc}
\toprule
Type & $\beta\equiv0\pmod3$ & $\alpha\equiv0\pmod3$ & $\alpha\beta\in\Z_3^\times$\\
\midrule
A & $C_3\oplus C_9$ & $C_3^2$ & $C_3^2$\\
B & $C_{27}$ & $C_3^2$ & $C_9$\\
C & $C_3$ & $C_3^2$ & $C_3$\\
\bottomrule
\end{tabular}
\end{table}

\begin{proof}
The curves of types A and C have CM by $\Z[\zeta_3]$; those of type B have CM by the order of conductor $3$. Their models have discriminants supported at $3$, so Corollary~\ref{cor:cm-sufficient} gives $E[3^\infty]\subset E(M)$.

Put $\pi=\sqrt{-3}$ and $\mathcal O=\Z_3[\pi]$, so $\pi^2=-3$. There is a direct product
\begin{equation}\label{eq:local-three}
 \mathcal O^\times=\mu_6\times U_2,
 \qquad U_2=1+3\mathcal O,
 \qquad \log U_2=3\mathcal O.
\end{equation}
The Weber-function form of the main theorem of complex multiplication \cite[Theorem~6.15]{CohenStevenhagen}, says that the division fields modulo global automorphisms generate the ray class fields. Thus for a maximal-order curve the projection of its CM image to $\mathcal O^\times/\mu_6\simeq U_2$ is surjective. As all division points lie in $M$, this projection is a surjective homomorphism from $\Gal(M/K)\simeq\Z_3^2$ to $U_2\simeq\Z_3^2$, hence an isomorphism. It respects conjugation.

Consider first the type A curve $E_0=[0,0,1,0,0]$. Its full $3$-torsion is defined over $K$, as follows from the division polynomial $3x(x^3+1)$ and the corresponding $y$-coordinates. Its image therefore lies in $1+3\mathcal O$ and is exactly $U_2$. Identifying $T_3(E_0)$ with $\mathcal O$, choose units $c,a\in\Z_3^\times$ so that the logarithms of the images of $\gamma_c,\gamma_a$ are $3c,3a\pi$, respectively. If $h=\beta\gamma_c-\alpha\gamma_a$ generates $\Gal(M/L)$, then
\[
 \log\rho(h)=3(c\beta-a\alpha\pi),
 \qquad \rho(h)-1=3(A+B\pi),
\]
where
\begin{equation}\label{eq:AB-congruences}
 A\equiv c\beta\pmod3,\qquad B\equiv-a\alpha\pmod3.
\end{equation}
The congruences follow from the convergent exponential series in $3\mathcal O$.

In the basis $1,\pi$, multiplication by $\rho(h)-1$ is
\begin{equation}\label{eq:max-matrix}
 \begin{pmatrix}3A&-9B\\3B&3A\end{pmatrix},
 \qquad \det=9(A^2+3B^2).
\end{equation}
If $\beta$ is a unit, $A$ is a unit, so the Smith invariants are $3,3$. If $\beta$ is divisible by $3$, then $\alpha$ and $B$ are units; the determinant has $3$-valuation $3$ and the Smith invariants are $3,9$. The fixed subgroup of $(\Q_3/\Z_3)^2$ under $h$ is isomorphic to the cokernel of this matrix. This gives the type A row. The second type A curve is the quadratic twist of $E_0$ by $-3$ and is isomorphic to it over $K$.

For type B, the rational $3$-isogeny with kernel polynomial $x+1$ on $E_0$ has quotient
\[
 E_1=[0,0,1,-30,63].
\]
The dual isogeny embeds $T_3(E_1)$ into $T_3(E_0)$ with index $3$. Its multiplier order is
\[
 \mathcal O'=\Z_3+3\mathcal O=\Z_3+\Z_3(3\pi).
\]
An index-three lattice in $\mathcal O$ is either the $\mathcal O$-ideal $\pi\mathcal O$, or the inverse image of a nonideal line in $\mathcal O/3\mathcal O$. The three nonideal lines are permuted transitively by $\mathcal O^\times$. The former lattice has maximal multiplier order and is excluded here. Thus, up to a scalar commuting with the CM action, the lattice is $\mathcal O'$. In the basis $1,3\pi$, the matrix becomes
\begin{equation}\label{eq:nonmax-matrix}
 \begin{pmatrix}3A&-27B\\B&3A\end{pmatrix}.
\end{equation}
Its determinant is the same as in \eqref{eq:max-matrix}. If $\alpha\equiv0$ and $\beta$ is a unit, the Smith invariants are $3,3$. If both coordinates are units, they are $1,9$. If $\beta\equiv0$ and $\alpha$ is a unit, they are $1,27$. This gives the type B row. The other type B curve is again a quadratic twist by $-3$.

For each type C curve, projection modulo $\mu_6$ again identifies the projective image with $U_2$. The CM image therefore has the form
\begin{equation}\label{eq:graph-character}
 \{\theta(v)v:v\in U_2\},
 \qquad \theta:U_2\longrightarrow\mu_3.
\end{equation}
There is no $\mu_2$ component because the image is pro-$3$. After completing the square, write a type C curve as $y^2=x^3+b$. The values of $-4b$ are $3,-9,-81,243$. None is a cube in the quadratic field $K$, so the factor $x^3+4b$ of its $3$-division polynomial does not split over $K$. Hence full $3$-torsion is not defined over $K$, and $\theta$ is nontrivial: otherwise the image would be $U_2$, which acts trivially on $E[3]$.

Compatibility with conjugation gives $\theta(\bar v)=\theta(v)^{-1}$. Under the logarithm in \eqref{eq:local-three}, $\theta$ is therefore trivial on the plus direction $3\Z_3$ and nontrivial on the minus direction $3\pi\Z_3$ modulo $9\pi\Z_3$. Writing $v(h)$ for the projection of $\rho(h)$ to $U_2$, we obtain
\[
 \theta(v(h))=1\quad\Longleftrightarrow\quad\alpha\equiv0\pmod3.
\]
If $\alpha$ is a unit, the element $\theta(v(h))v(h)-1$ has $\pi$-valuation $1$, because $v(h)-1\in3\mathcal O$ and a nontrivial cube root of unity differs from $1$ by an element of valuation $1$. The fixed group is therefore $\mathcal O/\pi\mathcal O\simeq C_3$. If $\alpha\equiv0$, then $\beta$ is a unit and $\theta(v(h))=1$; the valuation is $2$ and the group is $\mathcal O/3\mathcal O\simeq C_3^2$. This proves the type C row.

In each case $\rho(h)-1$ has nonzero determinant. Since $h$ topologically generates $\Gal(M/L)$, its fixed subgroup is $E(L)[3^\infty]$. This proves the table.
\end{proof}

\begin{remark}[Finite-layer checks]\label{rem:cm-finite-checks}
The torsion groups over the number fields used in the examples and the first CM layers were checked with PARI/GP~2.15.4. For the curves in Table~\ref{tab:cm-three}, the first two cyclotomic layers are $K_{\cyc}^{(1)}=\Q(\zeta_9)$ and $K_{\cyc}^{(2)}=\Q(\zeta_{27})$, with defining polynomials $X^6+X^3+1$ and $X^{18}+X^9+1$, respectively. The first anticyclotomic layer is $K(\sqrt[3]{3})$, as proved in Theorem~\ref{thm:hesse-family}. A mixed cubic layer is $K(\sqrt[3]{3\zeta_3})$, which has defining polynomial $X^6+3X^3+9$ over $\Q$. This extension of $K$ is cyclic and unramified outside $3$, so it lies in $M$ by Lemma~\ref{lem:ray-compositum}; it is distinct from the first layers of both $K_{\cyc}$ and $K_{\anti}$. The computed groups agree with Table~\ref{tab:cm-slopes}, whose proof determines the torsion on the full extensions from the local CM action.
\end{remark}

\begin{remark}\label{rem:cm-noncyc-27}
For type B, $E(L_{1,3})[3^\infty]=C_{27}$, whereas $L_{1,3}\cap K_{\cyc}=K_{\cyc}^{(1)}$. Thus the hypothesis $K\neq\Q(\sqrt{-3})$ in Proposition~\ref{prop:odd-three} is necessary. For types B and C, full $3$-torsion occurs precisely when $\alpha\equiv0\pmod3$, in accordance with Theorem~\ref{thm:full-three}.
\end{remark}

\section{Torsion in the compositum of \texorpdfstring{$\Z_2$}{Z2}-extensions}\label{sec:two}

Let $K$ be any quadratic field and put $F=\widetilde K_2$. Then $F/\Q$ is Galois and pro-$2$. The proofs below use these properties and the ramification bound \eqref{eq:inertia-bound}.

\subsection{Odd primes at least \texorpdfstring{$5$}{5}}

\begin{lemma}[Quadratic twists of cyclic torsion]\label{lem:two-character}
Let $q\geq5$ be prime and let $P\in E(F)$ have order $q^n$. There is a rational quadratic twist $E^d$ having a point of order $q^n$ over $\Q_{\infty,2}$.
\end{lemma}
\begin{proof}
The bound \eqref{eq:inertia-bound} excludes $\mu_q$ from $F$, so $E(F)[q^n]=\langle P\rangle$ is cyclic and $G_\Q$-stable. Let
\[
 \psi:G_\Q\longrightarrow(\Z/q^n\Z)^\times
\]
be its character. Its image is a $2$-group, because $\Q(P)\subset F$. By Kronecker--Weber, it is a Dirichlet character. For every odd prime $\ell$, its restriction to inertia has order at most $2$, by \eqref{eq:inertia-bound}. The local Dirichlet component at $\ell$ therefore has order at most $2$; it is trivial on the odd-order wild units and is either trivial or the quadratic character at $\ell$.

Let $\eta_0$ be the product of the inverses of the odd conductor components. It is a quadratic character, and $\eta_0\psi$ has conductor a power of $2$. The maximal abelian pro-$2$ extension of $\Q$ with such conductor is
\[
 \Q(\mu_{2^\infty})=\Q_{\infty,2}(i).
\]
Its Galois group is $C_2\times\Z_2$. A further twist by the quadratic character of $\Q(i)$, if necessary, removes the $C_2$ component: an element of order $2$ in $(\Z/q^n\Z)^\times$ is necessarily $-1$. Thus $\eta\psi$ cuts out a subfield of $\Q_{\infty,2}$ for a quadratic character $\eta$ of $\Q$. On the twist $E^d$ associated to $\eta$, the corresponding cyclic subgroup has character $\eta\psi$, yielding the required point.
\end{proof}

\begin{theorem}\label{thm:two-odd-large}
For every quadratic $K$,
\[
 E(F)[5^\infty]\in\{0,C_5\},\qquad
 E(F)[7^\infty]\in\{0,C_7\},
\]
and $E(F)[q^\infty]=0$ for every prime $q\geq11$.
\end{theorem}
\begin{proof}
Apply Lemma~\ref{lem:two-character} and the $p=2$ theorem of \cite{CDKN}. Its possible groups have no odd prime factors outside $3,5,7$, and no points of order $25$ or $49$. The Weil pairing gives cyclicity for the $5$- and $7$-primary parts.
\end{proof}

\begin{remark}[The primes $13$ and $17$ in Li's preprint]\label{rem:Li-two}
Theorem~\ref{thm:two-odd-large} excludes $q=17$, which is left open in \cite[Theorem~1.1(iii), Remark~5.6]{Li}. It also completes the argument for $q=13$. The exclusion list in \cite[Proposition~5.4]{Li} omits $13$, and the subsequent degree argument does not resolve this case: since $13-1=12$, a cyclic quartic point field is still possible, and the quartic Galois classification allows $C_{13}$. Thus the conclusion at $13$ is correct, but the proof has a gap. Lemma~\ref{lem:two-character} supplies the additional ramification argument needed for both primes.
\end{remark}

\subsection{The \texorpdfstring{$3$}{3}-primary part}

\begin{theorem}\label{thm:two-three}
For every quadratic $K$,
\[
 E(F)[3^\infty]\in\{0,C_3,C_9,C_3^2\}.
\]
\end{theorem}
\begin{proof}
First, $\mu_9\not\subset F$, since its ramification index at $3$ over $\Q$ is $6$, contrary to \eqref{eq:inertia-bound}. Thus the smaller invariant factor of any finite $3$-primary subgroup of $E(F)$ has order at most $3$.

If $E(F)[3]$ has rank at most one, a point $P$ of order $3^n$ generates a $G_\Q$-stable cyclic subgroup. The $2$-Sylow subgroup of $(\Z/3^n\Z)^\times$ has order $2$, so $\Q(P)$ has degree at most $2$. The quadratic-field torsion classification \cite{Najman} bounds $3^n$ by $9$.

Now suppose $E[3]\subset E(F)$, and assume for a contradiction that a point of order $9$ exists. The finite $G_\Q$-stable group
\[
 T=E(F)[9]
\]
is $C_3\oplus C_9$, because full $9$-torsion was excluded. Its characteristic subgroup $3T$ is a $G_\Q$-stable line in $E[3]$. Thus the mod-$3$ image is contained in a Borel subgroup of $\GL_2(\F_3)$, of order $12$. It is also a $2$-group, so
\[
 D=\Q(E[3])\subset F,\qquad [D:\Q]\leq4.
\]
The field $D$ is Galois over $\Q$. Since $D(E[9])/D$ has $3$-power degree and $F/D$ is pro-$2$, the subgroup $C_3\oplus C_9$ would already be defined over $D$. This is excluded over quadratic fields by \cite{Najman} and over quartic Galois fields by \cite[Theorem~1.2]{ChouQuartic}. This contradiction excludes points of order $9$ and hence all higher $3$-power torsion in the rank-two case.
\end{proof}

The group $C_3^2$ occurs for $E=[0,0,1,0,0]$ and $K=\Q(\sqrt{-3})$, since $E[3]$ is already defined over $K$.

\begin{corollary}[Finite control of the odd part]\label{cor:two-odd-control}
For every subfield $K\subset L\subset F$,
\[
 E(L)_{\tors}^{\odd}
 =E(L\cap K(E[105]))_{\tors}^{\odd}.
\]
\end{corollary}
\begin{proof}
Theorems~\ref{thm:two-odd-large} and~\ref{thm:two-three} leave only the primes $3,5,7$. Apply Lemma~\ref{lem:coprime} to each. The same lemma places every $9$-torsion point over $L$ in $K(E[3])$.
\end{proof}

\subsection{Non-CM \texorpdfstring{$2$}{2}-primary torsion}

\begin{theorem}\label{thm:noncm-two-bound}
Let $E/\Q$ be non-CM and let $K$ be quadratic. For every abelian extension $A/K$,
\[
 E(A)[2^\infty]\subset E[2^{12}].
\]
In particular, for $K\subset L\subset\widetilde K_2$,
\[
 E(L)[2^\infty]
 =E(L\cap K(E[2^{12}]))[2^{12}].
\]
\end{theorem}
\begin{proof}
By \cite[Corollary~1.3]{RZB}, the $2$-adic image over $\Q$ contains the principal congruence subgroup of level $32$. In particular it contains $I+32e_{12}$ and $I+32e_{21}$. The image of $G_K$ is normal of index at most $2$, so it contains their squares
\[
 I+64e_{12},\qquad I+64e_{21}.
\]
Lemma~\ref{lem:commutator}, with $t=64$, gives the assertion. The division-field equality follows by intersecting coordinate fields.
\end{proof}

Together with Corollary~\ref{cor:two-odd-control}, this theorem reduces the torsion of a non-CM curve over $\widetilde K_2$ to finite division fields.

\subsection{Infinite primary torsion and CM curves}

\begin{proposition}[Necessary conditions for infinite primary torsion]\label{prop:infinite-necessary}
Let $p=2$ or $3$. If $E(\M)[p^\infty]$ is infinite, then
\[
 E[p^\infty]\subset E(\M).
\]
Moreover, $E$ is CM, its CM field is $K$, and it has good reduction over $K$ outside $p$. For the three pairs in Lemma~\ref{lem:ray-compositum}, these latter conditions are also sufficient.
\end{proposition}
\begin{proof}
Write
\[
 T_n=E(\M)[p^n]\simeq C_{p^{a_n}}\oplus C_{p^{b_n}},
 \qquad 0\leq a_n\leq b_n.
\]
The characteristic cyclic subgroup $p^{a_n}T_n$ is $G_\Q$-stable, and gives a rational cyclic isogeny of degree $p^{b_n-a_n}$. By \eqref{eq:isogeny-list}, $b_n-a_n\leq4$ if $p=2$ and $b_n-a_n\leq3$ if $p=3$. Infinite primary torsion forces $b_n$ to be unbounded, hence $a_n$ is unbounded. A subgroup with both invariant factors divisible by $p^r$ contains $E[p^r]$. Since $r$ is arbitrary, $E[p^\infty]\subset E(\M)$.

Theorems~\ref{thm:noncm-three-bound} and~\ref{thm:noncm-two-bound} exclude non-CM curves. Let $k$ be the CM field. The $p$-adic image over $k$ is open in its Cartan subgroup, while elements not fixing $k$ act on that Cartan by the nontrivial quadratic automorphism; see \cite[Proposition~12.3]{RSZB}. If $k\neq K$, the image of $G_K$ contains both an open subgroup of the Cartan and such an element, and cannot be abelian. But $E[p^\infty]\subset E(\M)$ forces this image to factor through the abelian group $\Gal(\M/K)$. Therefore $k=K$.

Since the full $p$-adic representation is unramified outside $p$, the N\'eron--Ogg--Shafarevich criterion gives good reduction there. The sufficiency assertion is Corollary~\ref{cor:cm-sufficient}.
\end{proof}

\begin{corollary}[CM examples at $p=2$]\label{cor:sixteen}
For each pair $(E,K)$ in Table~\ref{tab:cm-two},
\[
 E(\widetilde K_2)[2^\infty]
 =E[2^\infty]\simeq(\Q_2/\Z_2)^2.
\]
Table~\ref{tab:cm-two} gives sixteen such pairs.
\end{corollary}

\begin{table}[htbp]
\centering
\caption{CM pairs with full $2$-power torsion over $\widetilde K_2$.}\label{tab:cm-two}
\begin{tabular}{llrc}
\toprule
$K$ & Coefficients & $j$ & Number\\
\midrule
$\Q(i)$ & $[0,0,0,a,0]$, $a\in\{\pm1,\pm2,\pm4,\pm8\}$ & $1728$ & $8$\\
$\Q(i)$ & $[0,0,0,-11,\pm14]$ & $287496$ & $2$\\
$\Q(i)$ & $[0,0,0,-44,\pm112]$ & $287496$ & $2$\\
$\Q(\sqrt{-2})$ & $[0,\varepsilon,0,-13,-21\varepsilon]$, $\varepsilon=\pm1$ & $8000$ & $2$\\
$\Q(\sqrt{-2})$ & $[0,\varepsilon,0,-3,\varepsilon]$, $\varepsilon=\pm1$ & $8000$ & $2$\\
\bottomrule
\end{tabular}
\end{table}

\begin{proof}
The CM discriminants corresponding to these three $j$-invariants are $-4$, $-16$, and $-8$. The displayed integral equations have discriminants that are powers of $2$ up to sign. Thus they have good reduction outside $2$ over $\Q$, and hence over $K$. Apply Corollary~\ref{cor:cm-sufficient}.
\end{proof}

\subsection{Full finite torsion and bad reduction}

\begin{proposition}\label{prop:full-four-family}
There are infinitely many rational $j$-invariants of non-CM elliptic curves $E/\Q$ with
\[
 E[4]\subset E(\Q(i)).
\]
\end{proposition}
\begin{proof}
Choose nonzero rational $a,b,c$ with $a^2+c^2=b^2$, and set
\[
 E_{a,b}:y^2=x(x-a^2)(x-b^2).
\]
Every difference between two roots of the cubic is a square in $\Q(i)$: the differences are $\pm a^2,\pm b^2,\pm c^2$. The standard halving formulas for full rational $2$-torsion therefore give $E_{a,b}[4]\subset E_{a,b}(\Q(i))$. Explicitly, halving $(e_i,0)$ requires the square roots of $e_i-e_j$ and $e_i-e_k$, and these all lie in $\Q(i)$.

Take $a=2s$, $c=s^2-1$, $b=s^2+1$ with rational $s\notin\{0,1,-1\}$. The Legendre parameter $a^2/b^2$ and the resulting $j$-invariant vary nontrivially with $s$. Only finitely many rational $j$-invariants are CM, so infinitely many non-CM values occur.
\end{proof}

For example, $a=3,b=5,c=4$ gives
\[
 E:y^2=x(x-9)(x-25),\qquad j(E)=\frac{111284641}{50625}.
\]
Its minimal equation is $[1,1,1,-10,-10]$, of conductor $15$, and $E(\Q(i))_{\tors}=C_4^2$. The negative valuations of $j$ at $3$ and $5$ show that $E$ has bad reduction at both primes after base change to $\Q(i)$, although $E[4]$ is defined over that field.

\begin{remark}
The N\'eron--Ogg--Shafarevich criterion requires the entire $\ell$-adic Tate module to be unramified, with $\ell$ different from the residue characteristic. Proposition~\ref{prop:infinite-necessary} meets this hypothesis by first proving containment of $E[p^\infty]$. A single finite module does not suffice: the example above and Example~\ref{ex:noncm-three} have unramified $E[4]$ or $E[3]$ at primes of bad reduction.
\end{remark}

\section{Further questions}\label{sec:scope}

The division-field bounds in Theorems~\ref{thm:noncm-three-bound} and~\ref{thm:noncm-two-bound} leave two finite classification problems: the $3$-primary torsion of non-CM curves over $\widetilde{\Q(\sqrt{-3})}_3$, and the $2$-primary torsion over $\widetilde K_2$. In the latter case, one must also determine which primary groups can occur together. Computing the relevant Galois invariants could sharpen the bounds and settle these questions for individual curves.

The CM tables treat the curves listed, without classifying all quadratic fields and twists. A complete description of infinite primary torsion would require determining all rational CM curves with good reduction outside $p$ over their CM field and applying Proposition~\ref{prop:infinite-necessary}.

\section*{Declarations}
\paragraph{Funding}
This research was supported by the National Natural Science Foundation of China (NSFC), Grant Nos.~12231009 and~11971224.
\paragraph{Conflict of interest}
The authors declare that they have no conflict of interest.
\paragraph{Data availability}
The curve models and defining fields used in the computations are included in the article.

\end{document}